\documentclass[12pt,psamsfonts]{article} 
\usepackage{a4,fullpage,amssymb,epsf,psfrag,times}
\usepackage{graphicx}
\usepackage{amsmath}
\usepackage{amsfonts}
\usepackage{enumerate}
\usepackage{latexsym}
\usepackage{epsfig}
\usepackage{amsthm}  
\usepackage{amsmath}
\usepackage{amssymb} 
\usepackage{latexsym}
\usepackage{epsfig} 
\usepackage{amssymb,latexsym}
\usepackage[all]{xy}
\usepackage{graphicx}
\usepackage{amsmath}
\usepackage{amsfonts}
\usepackage{enumerate}
\usepackage{latexsym}
\usepackage{epsfig}
\usepackage{amsthm}  
\usepackage{amsmath}
\usepackage{amssymb} 
\usepackage{latexsym}
\usepackage{epsfig}

\makeatletter\@addtoreset{equation}{section}\makeatother
\makeatletter\@addtoreset{table}{section}\makeatother

\newtheorem{conjecture}{Conjecture}[section]
\newtheorem{theorem}{Theorem}[section]
\newtheorem{prop}[theorem]{Proposition}
\newtheorem{lemma}[theorem]{Lemma}

\newtheorem{cor}[theorem]{Corollary}

\newenvironment{question}{\refstepcounter{theorem}\par\medskip\noindent{\bf Question~\thetheorem~~}}{\unskip\nobreak\hfill\hbox{ $\oslash$}\par\bigskip}

\newenvironment{example}{\refstepcounter{theorem}\par\medskip\noindent{\bf Example~\thetheorem~~}}{\unskip\nobreak\hfill\hbox{ $\oslash$}\par\bigskip}

\newenvironment{definition}{\refstepcounter{theorem}\par\medskip\noindent{\bf Definition~\thetheorem~~}}{\unskip\nobreak\hfill\hbox{ $\oslash$}\par\bigskip}

\begin{document}

\title{Topology of spaces\linebreak[1] of equivariant symplectic embeddings}
\date{}
\author{Alvaro Pelayo}

\maketitle

\begin{abstract}
We compute the homotopy type of the space of
$\mathbb{T}^n$\--equivariant symplectic embeddings from the standard
$2n$\--dimensional ball of some fixed radius into a $2n$\--dimensional
symplectic--toric manifold $(M, \, \sigma)$, and use this computation to define a
$\mathbb{Z}_{\ge 0}$\--valued step function on $\mathbb{R}_{\ge 0}$ which is an
\emph{invariant} of the symplectic--toric type of $(M, \, \sigma)$.
We conclude with a discussion of the partially equivariant case of this result.
\end{abstract}

\section{The main theorem}
Let $(M,\, \sigma)$ be a $2n$\--dimensional symplectic manifold and write
$\mathbb{B}_r$ for the compact $2n$\--ball of radius $r>0$ in the complex space
$\mathbb{C}^n$ equipped with the restriction of the standard symplectic form 
$\sigma_0$ of $\mathbb{C}^{n}$. (The proofs of the results in this paper hold
verbatim for the open ball.)
Recently a lot of effort has been put into understanding the topological and geometric properties 
of the space of symplectic embeddings from $\mathbb{B}_r$ into $M$.
This question is not only intriguing, but it is also very fundamental 
because it acknowledges one of the main differences that exist between Riemannian
and symplectic geometry, e.g. Gromov's non--squeezing theorem \cite{G}. 

\begin{figure}[htb]
\begin{center}
\includegraphics{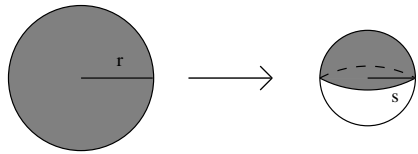}
\caption{An equivariant and symplectic embedding 
$\mathbb{B}^2_{r} \to \mathbb{S}^2_s$ with $r/s=\sqrt{2}$.}\label{AF5}
\end{center}
\end{figure}

This question, posed with such generality, has proven to be extremely difficult to answer.
Significant progress has been made by McDuff \cite{M1}, \cite{M2},
Biran \cite{B1}, \cite{B3}
and most recently by Lalonde--Pinsonnault \cite{LP}, among other authors. One of the most
general results is due to McDuff; she showed the connectedness of the space
of 4\--balls into 4\--manifolds with non-simple Seiberg--Witten type,
in particular rational or ruled surfaces. Recall that 
we say that a symplectic $4$\--manifold $Q$ has \emph{simple 
Seiberg--Witten type} or just \emph{simple type} 
if the only non--zero Gromov invariants of $Q$ occur 
in classes $A \in \textup{H}_2(Q)$ for which $k(A) = K \cdot A + A ^{2} = 0$.  
It follows from work of Taubes and Li--Liu that the symplectic $4$\--manifolds with 
non--simple type are blow--ups of (i) rational and ruled manifolds; 
(ii) manifolds with $\textup{b}_1 = 0$, $\textup{b} _2^{+} = 1$, like the Enriques or Barlow surface;
and (iii) manifolds with $\textup{b}_1 = 2$ and $(\textup{H}_1(X))^2\neq 0$;
examples (with $K=0$) are hyperelliptic surfaces, some 
non--K\"ahler $\mathbb{T}^2$\--bundles over $\mathbb{T}^2$ and
quotients $\mathbb{T}^{2} \times \Sigma_g / G$ where $\Sigma_g$  
is a surface of genus $g$ greater than $1$, and $G$ is certain finite group. 
See \cite{M1} for further references and examples.

McDuff's techniques are unique
to dimension 4 and do not extend at all to higher dimensions---this is also
the case in the other authors' work---existence of $J$\--holomorphic
curves with special homological properties is essential in their proofs.
Although $J$\--holomorphic curves exist in all even dimensions, it is
only in dimension $4$ where these homological properties hold. 

In the present paper we study a special case of this question: $M$ is a symplectic--toric
manifold of arbitrary dimension, and the symplectic embeddings that we consider
preserve the toric structure, see Figure 1. Precisely this means that there exists
an automorphism $\Lambda$ of the $n$\--torus $\mathbb{T}^n$ such that the following diagram commutes:
\begin{eqnarray} \label{mainlemma}
\xymatrix{ \ar @{} [dr] |{\circlearrowleft}
\mathbb{T}^{n} \times \mathbb{B}_r  \ar[r]^{\Lambda \times f}      \ar[d]^{  {\cdot} }  &  \mathbb{T}^{n} \times M 
                  \ar[d]^{\psi}   \\
                   \mathbb{B}_r  \ar[r]^f   &       M },
\end{eqnarray} 
where $\psi$ is a fixed effective and Hamiltonian $\mathbb{T}^n$\--action on $M$
and $ {\cdot}$ denotes the standard action by rotations
on $\mathbb{B}_r$ (component by component). In this case we say
that $f$ is a $\Lambda$\--\emph{equivariant mapping}.

\begin{figure}[htb]
\begin{center}
\includegraphics{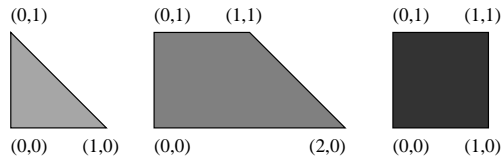}
\caption{The momentum polytope of $\mathbb{CP}^2$ and $\mathbb{B}^2_1$ (left),
of a Hirzebruch surface (center) and of $(\mathbb{CP}^1)^2$ (right).}
\end{center}
\label{AF3}
\end{figure}

The feature that makes the study of symplectic manifolds equipped with Hamiltonian
torus actions richer than the study of generic symplectic manifolds is the presence of the
smooth momentum map
$
\mu^{M} \colon M \to \mbox{Lie}(\mathbb{T}^n)^*,
$
whose image $\Delta^M$ is a convex polytope (called the \emph{momentum polytope
of $M$}, cf. Figure 2) as shown independently by Atiyah and Guillemin--Sternberg \cite{A}, 
\cite{G1}. Here we are
identifying the Lie algebra $\mbox{Lie}(\mathbb{T}^n)$ and its dual 
$\mbox{Lie}(\mathbb{T}^n)^*$ with $\mathbb{R}^n$.
Since this identification is not canonical, we need to specify 
the convention we adopt in this paper. This amounts to choosing
an epimorphism $ \mathbb{R} \to \mathbb{T}^1$ which we take to be
$x \mapsto e^{2\, \sqrt{-1}x}$. This epimorphism induces an isomorphism
between $\mbox{Lie}(\mathbb{T}^1)$ and $\mathbb{R}$ via
$\frac{\partial}{\partial x} \mapsto 1/2$, giving rise to a new isomorphism
$\mbox{Lie}(\mathbb{T}^n) \to \mathbb{R}^n$, $\frac{\partial}{\partial x_k} \mapsto 1/2 \, {e}_k$, 
by canonically identifying $\mbox{Lie}(\mathbb{T}^n)$ with the product of $n$ copies
of $\mbox{Lie}(\mathbb{T}^1)$ (see \cite{G2} for more details). 

For example, under the convention of the previous paragraph, 
the momentum map $\mu^{\mathbb{B}_r}$ of $\mathbb{B}_r$ is a mapping from
$\mathbb{B}_r$ into $\mathbb{R}^n$ with components 
$
\mu_k^{\mathbb{B}_r}( {z})=|z_k|^2$, for all 
integer $k$ with $1 \le k \le n$. (There
are a number of different conventions used in the literature, and our
choice is intended to give the \emph{simplest} formula for the momentum
map of $\mathbb{B}_r$.)
The simplest symplectic manifolds which admit Hamiltonian
effective torus actions are called symplectic--toric.

\begin{definition}
A \emph{symplectic--toric manifold} $M$,
also called a \emph{Delzant manifold},
is a compact connected symplectic manifold equipped 
with an effective Hamiltonian action of a torus of dimension 
half of the dimension of the manifold. 
In this case the momentum 
polytope $\Delta^{M}$ is called the \emph{Delzant polytope of} $M$.
\end{definition}

Symplectic--toric
manifolds were classified
by Delzant in \cite{D}. In particular, he showed 
that the momentum image of such a manifold
under the momentum map completely determines $M$
up to equivariant symplectomorphisms.

The main result
of this paper, Theorem \ref{maintheorem} below, describes the topology of the 
space of equivariant embeddings of symplectic balls into a symplectic--toric manifold.
We denote by $\chi(M)$ the Euler characteristic of $M$.

\begin{theorem} \label{maintheorem}
For every symplectic--toric $2n$\--manifold $M$ there is an associated
$\mathbb{Z}$\--valued non--increasing step function 
$\textup{Emb}_{(M,\,\sigma)} \colon \mathbb{R}_{\ge 0} \to [0,\,n! \, \chi(M)]$
such that for each $r \ge 0$ the space of equivariant symplectic embeddings from 
the $2n$\--ball $\mathbb{B}_r$ into $M$ is homotopically
equivalent to a disjoint union of $\textup{Emb}_{(M,\,\sigma)}(r)$ subspaces, each of which is homeomorphic to the 
$n$\--torus $\mathbb{T}^n$. 
\end{theorem}

As a matter of fact we can explicitly and easily read $\textup{Emb}_{(M,\,\sigma)}$ from the polytope $\Delta^M$:

\begin{example}
Let $(M, \, \sigma)$ equal the blow--up of $\mathbb{S}^2_{r_0} \times \mathbb{S}^2_{r_0}$
with $r_0=1/\sqrt{2}$ whose Delzant
polytope has vertices at $(0,\, 0)$, $(2, \, 0)$, $(2, \, 1)$, $(1, \, 2)$ and $(0, \, 2)$ \textup{(}see \textup{Figure 4}\textup{)}.
Then $\textup{Emb}_{(M, \, \sigma)}=10 \, \chi_{[0,1)} + 2 \, \chi_{[1,\sqrt{2})}$, where $\chi_A$ denotes the 
characteristic function of $A \subset \mathbb{R}$. We 
identify the $2$\--sphere of radius $r$ equipped with the standard area form
with $(\mathbb{CP}^1,\, 4r^2 \cdot \sigma_{\textup{FS}})$,
where $\sigma_{\textup{FS}}$ is the Fubini--Study form.
\end{example}

\begin{prop}\label{maintheorem2}
The function $\textup{Emb}_{(M,\,\sigma)}(r)$ given in \textup{Theorem \ref{maintheorem}} is an invariant of the symplectic--toric
type of $M$ and is given by the formula 
\begin{eqnarray} \label{k}
\textup{Emb}_{(M,\,\sigma)}(r)= n! \, \sum_{  {p} \in M^{\mathbb{T}^n} } \textup{c}_{ {p}}(r),
\end{eqnarray}
where for each fixed point $ {p} \in M$, $\textup{c}_{ {p}}(r)=1$ if 
the infimum of the $\textup{SL}(n, \, \mathbb{Z})$\--lengths of the edges of $\Delta^M$ meeting at 
$\mu^{M}( {p})$ is strictly greater than $r^2$, and $\textup{c}_{ {p}}(r)=0$ otherwise. 
\end{prop}

\begin{example}
\normalfont
Let $\sigma_{\textup{FS}}$ be the Fubini--Study form on $\mathbb{CP}^n$ and 
observe that $\mathbb{T}^n$ acts naturally on $(\mathbb{CP}^n,\, \lambda \cdot \sigma_{\textup{FS}})$, $\lambda>0$, with $n+1$ 
fixed points. The momentum
polytope is a tetrahedrum with vertices at $ {0}$ 
and $\lambda \,  {e}_i,$ where the $ {e}_i$
are the canonical basis vectors in $\mathbb{R}^{n}$. So if $M=\mathbb{CP}^n \times \mathbb{CP}^m$, 
the space of equivariant
symplectic embeddings from $\mathbb{B}_r$ into $M$ is homotopically equivalent to 
$$
\bigsqcup_{k=1}^{(n+m)!(n+1)(m+1)} \mathbb{T}^{n+m}
$$
if $r < \sqrt{\lambda}$, and it is empty otherwise.
\end{example}

The study of the space of symplectic embeddings is directly related to the study of the 
\emph{symplectic ball packing problem} cf. \cite{B2}, the equivariant version of which was
treated in \cite{P2}.

\section{Proof of Theorem \ref{maintheorem}}
\setcounter{equation}{2}
In this section we prove Theorem \ref{maintheorem}. For clarity, the proof is divided in three steps,
which we describe next.

We start by introducing the notation and making the following observations:
\begin{enumerate}
\item[i)]
Throughout the proof $\mathcal{R}$ will denote the space of rotations by matrices of the form
$(\delta_{i \, \tau(i)} \theta_{i\, j})_{i,\, j=1}^n$
with $\tau \in \textup{S}_n$ (the symmetric group) and $\theta_{i\, j} \in \mathbb{T}^1$, and $\mathcal{E}_{ {x}}^{\mathbb{T}^n}$ will denote the
space of equivariant symplectic embeddings $f$ from $\mathbb{B}_r$ into $M$ such that $f({0})= {p}$
and $\mu^M(p)=x \in \Delta^M$, equipped with the $\textup{C}^m$\--Whitney topology ($m \ge 0$). \emph{Throughout the present section we fix $f$}. 

Since each component of $\mathcal{R}$ 
is canonically identified with the $n$\--torus 
$\mathbb{T}^{n}$ (cf. Corollary \ref{lastc}), Theorem \ref{maintheorem}
amounts to prove that if $p \in M^{\mathbb{T}^n}$ ($M^{\mathbb{T}^n}$ denotes the 
$\mathbb{T}^n$\--fixed point set) is such that $\mu^{M}( {p})= {x}$, then the space 
$\mathcal{E}_{ {x}}^{\mathbb{T}^n}$ gets identified with $\mathcal{R}$ via a homotopy equivalence. 
\item[ii)]
Secondly let $\mathcal{B}_r^{\mathbb{T}^n}$ denote the space of equivariant symplectomorphisms
of $\mathbb{B}_r$ (again with respect to the $\textup{C}^m$\--Whitney topology). 

Recall that, for example, the
$\textup{C}^m$\--Whitney topology on $\mathcal{B}_r^{\mathbb{T}^n}$ is given by the
well--known norm
$$
\parallel \phi \parallel_{\textup{C}^m}=\max_{0 \le k \le m}
\sup_{ {z} \in \mathbb{B}_r} \parallel \textup{T}^k \phi
\parallel _{\mathfrak{M}(\mathbb{C})},
$$
where we are taking the norm $\parallel \cdot
\parallel _{\mathfrak{M}(\mathbb{C})} $ on
the right--hand side of this expression to be the canonical Euclidean norm
on the space of $n \times n$ matrices with complex entries.
\item[iii)]
We identify the automorphism group $\mbox{Aut}(\mathbb{T}^n)$ with the 
matrix group\linebreak 
$\textup{GL}(n, \, \mathbb{Z})$.
\item[iv)]
The elements $ {\alpha}_i^{ {p}}$'s, $1 \le i \le n$, denote the weights of the 
isotropy representation of $\mathbb{T}^n$ on $\textup{T}_{ {p}}M$; the
canonical basis vectors $ {e}_i \in \mathbb{R}^n$
represent the weights of the isotropy representation of
$\mathbb{T}^n$ on $\textup{T}_{ {0}} \mathbb{B}_r$. 
\end{enumerate}

\vspace{1mm}

\emph{Step 1}: Invariance of the image $f(\mathbb{B}_{r})$.

\vspace{1mm}

In this step we first show how to go from smooth maps on manifolds to affine maps on polytopes (see diagram (\ref{ml})), 
and secondly we use this to show the invariance of the image $f(\mathbb{B}_{r}) \subset M$. Precisely, one can think 
of an embedding being equivariant in the sense of commuting with the $\mathbb{T}^n$\--action, 
and it is when we reparametrize the torus that $\Lambda$ appears.

\begin{lemma} \label{2.1}
Let $g$ be any $\Lambda$\--equivariant and symplectic embedding such that the normalization condition $f( {0})=g( {0})= {p}$ holds.
Then for all $ {z} \in \mathbb{B}_{r}$, if $\mathbb{T}^n \cdot z$ denotes the $\mathbb{T}^{n}$\--orbit
that passes through $ {z}$, the identity $f(\mathbb{T}^n \cdot z)=g(\mathbb{T}^n \cdot z)$ holds, and therefore $f(\mathbb{B}_r)=g(\mathbb{B}_r)$. 
\end{lemma}
\begin{proof}
Let $f,\, g$ be $\Lambda$\--equivariant and symplectic
embeddings from $\mathbb{B}_r$ into $M$ with $f( {0})=g( {0})= {p}$. Under the identifications described in Section 1, the following diagram commutes,
where the top arrow stands for the affine map with
linear part $(\Lambda^{\textup{t}})^{-1}$, which takes $ {0}$ to $ {x}$:
\begin{eqnarray} \label{ml} 
\xymatrix{ \ar @{} [dr] |{\circlearrowleft}
\Delta^{\mathbb{B}_r} \ar[r]^{(\Lambda^{\textup{t}})^{-1}+ {x}}      &  \Delta^M  \\
           \mathbb{B}_r  \ar[r]^f  \ar[u]^{\mu^{\mathbb{B}_r}}  &       M      \ar[u]_{\mu^M}}.
\end{eqnarray}
In order to prove the commutativity of diagram (\ref{ml}), 
we denote by $\xi_M$ the vector field induced by the element 
$\xi \in \textup{Lie}(\mathbb{T}^n)$ via the exponential map
and note that from the definition of the momentum maps $\mu^{M}$ and $\mu^{\mathbb{B}_{r}}$,
the $\Lambda$\--equivariance of $f$, and the fact that $f^*\sigma=\sigma_0$, 
we have the following sequence of equalities, where $\textup{T}_{f(z)}\mu^M$ and
$\textup{T}_z \mu^{\mathbb{B}_r}$ denote, respectively, the tangent mapping of 
$\mu^M$ at $f(z)$ and of $\mu^{\mathbb{B}_r}$ at $z$,
\begin{eqnarray}
\langle \textup{T}_{z} \mu^{\mathbb{B}_{r}} (v),\, \xi \rangle_{z} 
 &=& (\sigma_{0})_{z}({v},\, \xi_{\mathbb{B}_r} ({z})) \nonumber \\
&=&  \sigma_{f({z})}(\textup{T}_{z} f({v}),\, \Lambda(\xi)_{M}(f({z}))) 
\nonumber \\
&=& \langle \textup{T}_{f({z})}  \mu^{M} (\textup{T}_{z}f({v})),\, \Lambda(\xi) 
\rangle_{f(z)} \nonumber \\
&=& \langle \Lambda^{\textup{t}} \circ \textup{T}_{f(z)}  
\mu^{M}(\textup{T}_{z}f(v)),\, \xi \rangle_{z},  \label{eq1}
\end{eqnarray}
where $ {z} \in \mathbb{B}_{r}$, $ {v} \in \textup{T}_{ {z}} \mathbb{B}_{r}$ 
and $\xi \in \mbox{Lie}(\mathbb{T}^{n})$.
Therefore by equation (\ref{eq1}) and by using the chain rule we 
obtain that for all $ {z} \in \mathbb{B}_r$
and $ {v} \in \textup{T}_{ {z}} \mathbb{B}_r$
\begin{eqnarray}
\textup{T}_{ {z}} \mu^{\mathbb{B}_{r}} ( {v})=\textup{T}_{ {z}}  (\Lambda^{\textup{t}}
\circ \mu^{M} \circ f) ( {v}). \label{eq2}
\end{eqnarray}
Considering equation (\ref{eq2}), 
$f( {0})= {p}$ and $\mu^{M}( {p})= {x}$ and composing with $(\Lambda^{\textup{t}})^{{-1}}$,
after integration we obtain the commutativity condition  on diagram
(\ref{ml}). Notice that
diagram (\ref{ml}) also holds for the embedding $g$.

Then it follows from the conjunction of diagram (\ref{ml}) and diagram (\ref{mainlemma})
that for all $t \in \mathbb{T}^{n}$ the following identities hold:
\begin{eqnarray}
\mu^{M}(\psi(\Lambda(t),\, f( {z})))&=&\mu^{M}(f(t \cdot  {z})) \nonumber \\
                                                                 &=&(\Lambda^{\textup{t}})^{-1} \circ \mu^{\mathbb{B}_{r}}( {z})+  {x} \nonumber \\
                                                                 &=&
\mu^{M}(g(t \cdot  {z}))=\mu^{M}(\psi(\Lambda(t),\, g( {z}))). 
\label{eq5}
\end{eqnarray}
Expression (\ref{eq5}) is clearly equivalent to $\mu^{M}(\mathbb{T}^n \cdot f(z))=\mu^{M}(\mathbb{T}^n
\cdot g(z))$, since
$\Lambda$ is an automorphism. Now since $M$ is symplectic--toric, by the proof of the Atiyah--Guillemin--Sternberg convexity theorem
we know that each fiber of the momentum map $\mu^{M}$ consists of a single connected orbit
which together with the last equality implies that $\mathbb{T}^n \cdot f(z)=\mathbb{T}^n \cdot
g(z)$. Since
$f(\mathbb{B}_{r})$ is the union of the orbit images
$f(\mathbb{T}^n \cdot z)$, we immediately obtain
that $f(\mathbb{B}_r)=g(\mathbb{B}_r)$, which concludes the proof. 
\end{proof}

\begin{figure}[htb]
\begin{center}
\includegraphics{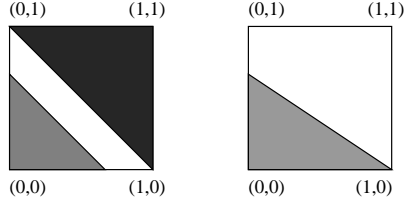}
\caption{Equivariant symplectic ball embeddings in $(\mathbb{CP}^1)^2$ (left);
the triangle on the right does \emph{not} come from such embedding.}
\end{center}
\label{AF2}
\end{figure}

It is possible to explicitly describe the momentum image $\mu^M(f(\mathbb{B}_r))$,
and for this purpose we recall the notion of $\textup{SL}(n, \, \mathbb{Z})$--length: if 
$ {x}, {y} \in \mathbb{R}^{n}$, we say that a segment line $[ {x}, \, {y}]$ joining
$ {x}$ to $ {y}$ in $\mathbb{R}^{n}$ 
\emph{has $\textup{SL}(n, \, \mathbb{Z})$\--length $d$}
if there exists a matrix $A \in \textup{SL}(n, \, \mathbb{Z})$ such that 
$A(d \,  {e}_{1})= {y}- {x}$ ($d$ is 
not defined in general, only for segments of rational slope). The Euclidean
length of a segment line agrees with its $\textup{SL}(n, \, \mathbb{Z})$\--length 
if and only if the segment is parallel to one 
of the coordinate axes in $\mathbb{R}^n$.

In their article \cite{KT2}, Karshon and Tolman made the following two definitions.
Let $(Q, \, \sigma^Q)$ be a connected symplectic $2m$\--dimensional 
manifold with momentum map $\mu^Q$ for an action of an $m$\--torus 
$\mathbb{T}^m$ on $Q$, and let $\Gamma \subset (\textup{Lie}(\mathbb{T}^m))^*$ 
be an open convex subset which contains the image of $Q$ under the momentum map $\mu^Q$. The
quadruple $(Q,\, \sigma^Q,\, \mu^Q,\, \Gamma)$ is a \emph{proper Hamiltonian $\mathbb{T}^m$\--
manifold} if the momentum map $\mu^Q$ is proper as a map to $\Gamma$. 

The proper Hamiltonian $\mathbb{T}^m$\--manifold $(Q,\, \sigma, \, \mu^Q,\, \Gamma)$
is said to be 
\emph{centered about a point $\alpha \in \Gamma$} if 
$\alpha$ is contained in the momentum map image of every component of $Q^K$,
for each $K \subset \mathbb{T}^m$.
Here 
$$
Q^K:=\{q \in Q \, | \, \psi_Q(a,\,q)=q, \, \, \forall a \in K\},
$$ 
where $\psi_Q \colon \mathbb{T}^m \times Q \to Q$ denotes the action of $\mathbb{T}^m$ on $Q$. The following lemma 
is Proposition 2.8 in \cite{KT2}.

\begin{lemma} [Karshon--Tolman, \cite{KT2}] \label{KTlemma}
Let  the quadruple $(Q,\, \sigma^Q,\, \mu^Q,\, \Gamma)$ be a proper Hamiltonian $2m$\--dimensional $\mathbb{T}^m$\--manifold. Suppose that $(Q,\, \sigma^Q,\, \mu^Q,\, \Gamma)$ is centered about  $\alpha \in \Gamma$ and that the preimage $(\mu^Q)^{-1}(\{\alpha\})$ consists
of a single fixed point $q$. Then $Q$ is equivariantly symplectomorphic to
$$
\{z \in \mathbb{C}^m \, | \, \alpha + \sum_{j=1}^m |z_j|^2 \, \eta^q_j \in \Gamma\},
$$
where $\eta^q_1,\ldots,\eta^q_m$ are the weights of the isotropy representation of $\mathbb{T}^m$ 
on $\textup{T}_qQ$. 
\end{lemma}

We use Lemma \ref{KTlemma}  in order to prove the following lemma.

\begin{lemma} \label{ss}
The momentum image $\mu^M(f(\mathbb{B}_r))$ equals the subset of $\mathbb{R}^n$ given by 
the convex hull of $x$ and  $x+r^2\, \alpha^{p}_{i}$, where $1 \le i \le n$.
Furthermore, the infimum of the $\textup{SL}(n,\, \mathbb{Z})$\--lengths of the edges of 
$\Delta^M$ meeting at  $x$ is greater than or equal to $r^2$,
if and only if  for all $0<s<r$ there exists an embedding $h \colon \mathbb{B}_s \to M$ 
which is $\Lambda$\--equivariant and symplectic, satisfying $h(0)=p$.
\end{lemma}

\begin{proof}
The first observation is that $\Delta^{\mathbb{B}_r}$ 
equals the convex hull in $\mathbb{R}^n$
of $0$ and $r^2 \, e_1,\ldots, r^2 \, e_n$. Secondly, since $\mu^{\mathbb{B}_r} \colon
\mathbb{B}_r \to \Delta^{\mathbb{B}_r}$ is onto, 
it follows from diagram (\ref{ml}) that 
$$
\mu^M(f(\mathbb{B}_r))=(\Lambda^{\textup{t}})^{-1}(\Delta^{\mathbb{B}_r})+x.
$$
Since $\Lambda$ is an automorphism,
$(\Lambda^{\textup{t}})^{-1}$ is an automorphism of the corresponding dual spaces and therefore
there exists a permutation $\tau \in \textup{S}_n$ such that $(\Lambda^{\textup{t}})^{-1}(e_i)=\alpha_{\tau(i)}^{p}$.
Then the linearity of $\Lambda$ implies that
$\mu^M(f(\mathbb{B}_r))$ equals the the convex hull in $\mathbb{R}^n$ of 
the points $x$ and $x+r^2 \, \alpha^{p}_{i}$, $1 \le i \le n$,
which proves the first claim.

Suppose that the infimum of the
$\textup{SL}(n,\, \mathbb{Z})$\--lengths
of the edges meeting at $x$ is greater than or equal to $r^2$. Let
$\Sigma$ be the convex hull of $x$ and $x+r^2\, \alpha_{i}^{p}$, with $1 \le i \le n$, and let 
$Z$ be the convex hull of $x+r^2\, \alpha_{i}^{p}$, with $1 \le i \le n$.
Notice that $\Sigma \subset \Delta^M$, $\Sigma \setminus Z$ is 
open in $\Delta^M$, and let
$\Gamma \subset \mathbb{R}^n$ be the open half--space of $\mathbb{R}^n$, whose 
closure's boundary $\partial (\textup{cl}(\Gamma))$ is the hyperplane of $\mathbb{R}^n$ that contains $Z$, and such that $\Sigma \setminus Z
\subset \Gamma$.

Let $N:=(\mu^M)^{-1}(\Sigma \setminus Z)$ and let $\sigma^N$ be
the symplectic form obtained by restricting $\sigma$ to $N$. The set $N$ is open in $M$ because 
it is the preimage of
the open set $\Sigma \setminus Z$ under the momentum map $\mu^M \colon
M \to \Delta^M$. By the proof of Atiyah--Guillemin--Sternberg convexity theorem, cf. \cite{A}, \cite{G1},
$N$ is a connected manifold.  Since $M$ is compact, the
momentum map $\mu^M \colon M \to \Delta^M$ is a proper map and therefore its restriction
$\mu^M \colon N \to \Sigma \setminus Z$ is a proper map, which means that
$\mu^M \colon N \to \Gamma$ is proper, since $(\mu^M)^{-1}(\Gamma \setminus (\Sigma \setminus Z))=\emptyset$. 
Therefore $(N,\, \sigma^N,\, 
\psi^{N})$ is a connected symplectic manifold with momentum map $\mu^M$,
and the quadruple $(N,\,\sigma^N,\,\mu^{M},\, \Gamma)$
is a proper Hamiltonian $\mathbb{T}^n$\--space.

On the other hand, notice that the quadruple $(N,\,\sigma^N,\,\mu^{M},\, \Gamma)$
is centered about the point $x$, and $(\mu^M)^{-1}(\{x\})=p$, so we can apply Lemma \ref{KTlemma}, 
and conclude that $N$ is equivariantly 
symplectomorphic
to the submanifold $X \subset \mathbb{C}^n$ given by
\begin{eqnarray}
X&:=&\{z \in \mathbb{C}^n \, | \, x + \sum_{i=1}^n |z_i|^2 \, \alpha_i^p \in \Sigma \setminus 
Z \} \nonumber \\
&=&\{z \in \mathbb{C}^n \, | \, x + \sum_{i=1}^n |z_i|^2 \, \alpha_i^p \in \Sigma\}
\setminus \{z \in \mathbb{C}^n \, | \, x + \sum_{i=1}^n |z_i|^2 \, \alpha_i^p \in 
Z\} \nonumber \\
&=& \mathbb{B}_r \setminus \partial \mathbb{B}_r=\textup{Int}(\mathbb{B}_r). \nonumber
\end{eqnarray}
Hence there exists an equivariant symplectomorphism $\phi:\textup{Int}(\mathbb{B}_r) \to 
N$, and by letting $j \colon \mathbb{B}_s \to \textup{Int}(\mathbb{B}_r)$ be 
the standard inclusion, if $s<r$, the map $h:=j_N \circ \phi \circ j : \mathbb{B}_s \to M$,
where $j_N \colon N \to M$ is the inclusion map, 
is an equivariant symplectic embedding for all $s<r$ with $h(0)=p$.
The converse follows from the first statement of the lemma. 
\end{proof}

Note that $\mu^M(f(\mathbb{B}_r))$ only depends on the fixed point
$ {p}$ and the radius $r$ (which was fixed a priori) and not on $f$. In
Figure 3 several momentum ball images are drawn using Lemma \ref{ss}. Note
that the shaded triangle on the right picture is not a Delzant
polytope since it fails to be smooth at $(0,\, 0)$. Delzant polytopes are simple,
edge--rational and smooth
polytopes, cf. Figure 2 (see \cite{G3} or \cite{C} for a definition of these notions). 

\vspace{1mm}

\emph{Step 2}: A deformation retraction on $\mathcal{B}_r^{\mathbb{T}^{n}}$.

\vspace{1mm}

In this step we use \emph{Alexander's trick} to construct a deformation retraction 
from the space of equivariant symplectomorphisms of the 
$2n$\--dimensional ball $\mathbb{B}_{r}$ in $\mathbb{C}^{n}$
onto a disjoint union of copies of $\mathbb{T}^{n}$.
The continuity of this deformation 
is standard and may be found in \cite{H}.

\begin{lemma} \label{2.3}
The space $\mathcal{B}_{r}^{\mathbb{T}^{n}}$ of
equivariant symplectomorphisms of the $2n$\--dimen-\linebreak sional ball 
$\mathbb{B}_{r}$
in $\mathbb{C}^{n}$, with respect to the standard symplectic form $\sigma_{0}$ and the canonical
action of $\mathbb{T}^{n}$ by rotations, deformation retracts onto its subspace of linear, equivariant and symplectic
rotations given by matrices in $\mathcal{R}$.
\end{lemma}

\begin{proof}
We define the transformation  $ {H}_{\mathbb{B}_r}^{\mathbb{T}^n}$ from  $\mathcal{B}_r^{\mathbb{T}^n} \times [0, \, 1] $
into  $\mathcal{B}_r^{\mathbb{T}^n}$, by the formula $ {H}_{\mathbb{B}_r}^{\mathbb{T}^n}(\phi,\, t):=\phi_{t}$, 
where $\phi_t$ is  the composite map 
\begin{eqnarray} \label{d3}
\phi_t:=(m_t)^{-1} \circ \phi \circ m_t, \textup{  }t \neq 0.
\end{eqnarray}
The map $m_t$ in expression (\ref{d3}) denotes the linear contraction of factor $0 \le t \le 1$ on $\mathbb{B}_r$, $m_t( {z})=t \,  {z}$;
and when $t=0$, $\phi_t=\phi_{0}$ is defined to be the tangent mapping $\textup{T} \phi$
of the map $\phi$, evaluated at $ {0}$. (This expression for $\phi_t$ is 
known as \emph{Alexander's trick}.) 
It is easy to check that $ {H}_{\mathbb{B}_{r}}^{\mathbb{T}^{n}}$
is continuous and that the evaluation map $[0, \,1] \times \mathbb{B}_r \to \mathbb{B}_r$
given by $(t, \, {z}) \mapsto \phi_t( {z})$ is smooth. Since 
$ {H}_{\mathbb{B}_{r}}^{\mathbb{T}^{n}}$
is the identity on linear maps, we conclude that it
is a deformation retraction, not only a homotopy, onto
the space of ball rotations by matrices 
$
(\delta_{i \, \tau(i)} \theta_{i\, j})_{i,\, j=1}^n
$
with $\tau \in \textup{S}_n$ (the symmetric group) and $\theta_{i\, j} \in \mathbb{T}^1$.

We have left to check that $ {H}_{\mathbb{B}_{r}}^{\mathbb{T}^{n}}$ is well defined,
i.e. that $\phi_{t} \in \mathcal{B}_{r}^{\mathbb{T}^{n}}$.  Indeed, the equivariance of 
the mapping $\phi_{t}$ follows directly from formula (\ref{d3}); explicitely we have that
if $\phi$ is equivariant with respect to $\Lambda \in \textup{Aut}(\mathbb{T}^n)$, then
$\phi_{t}(s \cdot  {z})= 1/t \, \phi(s \cdot t  {z})=\Lambda(s) \cdot \phi_{t}( {z})$
for all $s \in \mathbb{T}^{n}$. By differentiating formula (\ref{d3}) we
obtain that
\begin{eqnarray} \label{fh}
\textup{T} (\phi_t)=(m_t)^{-1} \circ \textup{T} \phi \circ m_t,
\end{eqnarray}
and since $m_t$ is a linear isomorphism, the mapping $\phi_{t}$ is a diffeomorphism. Furthermore, 
since the mapping $\phi$ is symplectic, it follows from expression (\ref{fh}) that
for all $ {z} \in \mathbb{B}_r$ we have that
$(\phi_{t}^* \sigma_0)_{z} (u,\,v)=
(\sigma_0)_{\phi(z)}(\textup{T}_z\phi(u), \, \textup{T}_z\phi(v))=(\sigma_0)_z(u,\,v)
$,
for every pair of vectors $u, \, v \in \textup{T}_z\mathbb{B}_r$,
and hence $\phi_t$ is a symplectic mapping. Therefore $\phi_t$ is a diffeomorphism, which is equivariant
and symplectic, or equivalently $\phi_t \in \mathcal{B}_{r}^{\mathbb{T}^{n}}$. We have been 
assuming that $t \neq 0$, but if $t=0$, it is trivial
that $\phi_{0} \in \mathcal{B}_r^{\mathbb{T}^n}$.
\end{proof}

\begin{cor} \label{lastc}
The space $\mathcal{B}_{r}^{\mathbb{T}^{n}}$ of
equivariant symplectomorphisms of the $2n$\--dim\--\linebreak ensional ball $\mathbb{B}_{r}$
in $\mathbb{C}^{n}$, with respect to the standard symplectic form $\sigma_{0}$ and the canonical
action of $\mathbb{T}^{n}$ by rotations, is homotopically equivalent
to a disjoint union of $n!$ copies of $\mathbb{T}^n$.
\end{cor}
\begin{proof}
Apply Lemma \ref{2.3} and observe that 
the space of ball rotations by matrices 
$
(\delta_{i \, \tau(i)} \theta_{i\, j})_{i,\, j=1}^n
$
with $\tau \in \textup{S}_n$ and $\theta_{i\, j} \in \mathbb{T}^1$
is homotopically equivalent to a disjoint union of $n!$ copies of $\mathbb{T}^n$.
\end{proof}

We conclude the proof with Step 3, in which Lemma \ref{2.1}, Lemma \ref{ss} and Lemma \ref{2.3}
are combined in order to prove Theorem \ref{maintheorem}. The proof
of Proposition \ref{maintheorem2} will follow from the proof
of Theorem \ref{maintheorem}, since the function $\textup{Emb}_{(M,\,\sigma)}$ will
be explicitly computed.

\vspace{1mm}
                                                                                                                            
\emph{Step 3}: Lifting the deformation $\phi_{t}$ to $\mathcal{E}_{ {x}}^{\mathbb{T}^{n}}$ and conclusion.

\vspace{1mm}

In this final step we show that $\mathcal{E}_{ {x}}^{\mathbb{T}^{n}}$
is homotopically equivalent to a disjoint union
of copies of $\mathbb{T}^{n}$.

\begin{lemma}
Suppose that the infimum of the $\textup{SL}(n, \, \mathbb{Z})$\--lengths of the edges of $\Delta^M$ meeting at 
$ {x}$ is strictly greater than $r^2$. Then there exists an
equivariant and symplectic embedding $u \colon \mathbb{B}_{r} \to M$
with $u({0}) = {p}$ such that if
$\rho$ is the identification map on $\mathcal{E}_{ {x}}^{\mathbb{T}^{n}}$ which takes
values on $\mathcal{B}_r^{\mathbb{T}^{n}}$ and is given 
by formula $\rho(h):=u^{-1} \circ h$, where $h \in \mathcal{E}_{ {x}}^{\mathbb{T}^{n}}$,
the space $\mathcal{E}_{ {x}}^{\mathbb{T}^{n}}$ is homotopically equivalent to 
the space $\rho^{-1}(\mathcal{R})$.
\end{lemma}
\begin{proof} The first observation is that by Lemma \ref{ss} there exists a
$\Lambda$\--equivariant and symplectic embedding $u$ from $\mathbb{B}_{r}$
into $M$ with $u({0}) = {p}$. In order to construct homotopy
equivalences between $\mathcal{E}_{ {x}}^{\mathbb{T}^{n}}$ and $\mathcal{R}$, we
define $\rho$ to be the identification map on $\mathcal{E}_{ {x}}^{\mathbb{T}^{n}}$
which takes values in $\mathcal{B}_r^{\mathbb{T}^{n}}$ and is given  by
formula $\rho(h):=u^{-1} \circ h$ for every $h \in
\mathcal{E}_{ {x}}^{\mathbb{T}^{n}}$.  Now we claim that the map
$ {H}^{\mathbb{T}^n}_{ {x}}$ from $\mathcal{E}_{ {x}}^{\mathbb{T}^n}
\times [0, \,1]$ to $\mathcal{E}_{ {x}}^{\mathbb{T}^n}$, given by the commutative diagram
(\ref{important}) below, is a well--defined and continuous homotopy  satisfying
$ {H}_{ {x}}^{\mathbb{T}^n}(\mathcal{E}_{ {x}}^{\mathbb{T}^n} \times
\{0\})= \rho^{-1}(\mathcal{R})$, while $\mathcal{E}_{ {x}}^{\mathbb{T}^{n}}$ is
preserved at time $t=1$, i.e. we have that 
$ {H}_{ {x}}^{\mathbb{T}^n}(\mathcal{E}_{ {x}}^{\mathbb{T}^n}  \times
\{1\})=\mathcal{E}_{ {x}}^{\mathbb{T}^n}$. The diagram is the following:
\begin{eqnarray} \label{important} \xymatrix{\ar @{} [dr] |{\circlearrowleft}
\mathcal{E}_{ {x}}^{\mathbb{T}^n} \times [0, \, 1]  
\ar[r]^{ \, \, \, \, \, \, \, \, \, \, \, {H}_{ {x}}^{\mathbb{T}^n} }   \ar[d]_{\rho \times
\mathrm{id}} & \mathcal{E}_{ {x}}^{\mathbb{T}^n}  \ar[d]^{\rho} \\
\mathcal{B}_r^{\mathbb{T}^{n}} \times [0, \,1]   
\ar[r]^{ \, \, \, \, \, \, \, \, \, \, \, {H}_{\mathbb{B}_r}^{\mathbb{T}^n}}  & 
\mathcal{B}_r^{\mathbb{T}^n} }.  
\end{eqnarray} 
The mapping $ {H}^{\mathbb{T}^n}_{ {x}}$ is well defined 
by Lemma \ref{2.1}.  Note that $ {H}_{ {x}}^{\mathbb{T}^n}$
is continuous,
since the identifications $\rho$ and $\rho^{-1}$ are obviously continuous and we showed in Lemma \ref{2.3}
that $ {H}_{\mathbb{B}_r}^{\mathbb{T}^n}$  is continuous.
We can therefore conclude, from the previous considerations and the fact that
that $ {H}_{\mathbb{B}_{r}}^{\mathbb{T}^{n}}$ is a deformation retraction in the $\textup{C}^m$\--Whitney topology,
that $ {H}^{\mathbb{T}^n}_{ {x}}$ induces homotopy equivalences $\rho$ and 
$\delta(f):= {H}_{ {x}}^{\mathbb{T}^n}(f,\, 0)$ between $\mathcal{E}_{ {x}}^{\mathbb{T}^{n}}$ and $\rho^{-1}
(\mathcal{R})$,
with $\rho \circ \delta$ homotopic to $\mbox{id}_{\mathcal{E}_{ {x}}^{\mathbb{T}^n} }$ and $\delta \circ \rho =\mbox{id}_{\rho^{-1}
(\mathcal{R})}$.
\end{proof}

In order to conclude the proof of Theorem \ref{maintheorem} we simply make the following observations:
\begin{itemize}
\item
First, the space described in it is precisely the 
disjoint union of the $\mathcal{E}_{ {x}}^{\mathbb{T}^n}$, $ {x}$ being a vertex of $\Delta^M$, 
because $ {0}$ is to be mapped to a fixed point of $\psi$. 
\item The number of $\mathbb{T}^n$\--fixed points, which is the same as the
number of vertices of $\Delta^M$, is precisely $\chi(M)$. This follows
from the analysis of the momentum map as in
Atiyah--Delzant--Guillemin--Sternberg theory (see for example \cite{G3},
\cite{G4}).
\item
If we denote by $\textup{Emb}_{(M,\,\sigma)}(r)$ the number of copies of $\mathbb{T}^{n}$ onto which the space considered in Theorem \ref{maintheorem}
retracts (see formula (\ref{k})), $\textup{Emb}_{(M,\,\sigma)}(r)$ is obtained by multiplying the number of fixed points that admit such an embedding (see Lemma \ref{ss})
by the number of copies of $\mathbb{T}^{n}$
onto which $\mathcal{E}_{ {x}}^{\mathbb{T}^{n}}$ (for the particular point)
retracts; this latter number is $n!$ (see Corollary \ref{lastc}), i.e. as many
copies of $\mathbb{T}^{n}$ as possible ways that the canonical basis vectors
$ {e}_{i}$ may be mapped onto the basis of  weights
$ {\alpha}_{i}^{ {p}}$ (for the particular point).  Also, the former
number is by Lemma \ref{ss} controlled by the Boolean variable
$\textup{c}_{ {p}}(r)$ defined in Proposition \ref{maintheorem2}. Therefore
$\textup{Emb}_{(M,\,\sigma)}(r)$ is given by $\textup{Emb}_{(M,\,\sigma)}(r)= n! \,\sum_{  {p} \in
M^{\mathbb{T}^n} } \textup{c}_{ {p}}(r)$, as we wanted to show.  
\item
It is obviously true that if $(M, \, \sigma)$ is equivariantly symplectomorphic to
$(\widetilde{M}, \, \widetilde{\sigma})$, then $\textup{Emb}_{(M,\,\sigma)}(r)=\textup{Emb}_{(\widetilde{M},\,\widetilde{\sigma})}(r)$, so the
integer $\textup{Emb}_{(M,\,\sigma)}(r)$ is a symplectic--toric invariant.
\end{itemize}

\begin{figure}[htb]
\begin{center}
\includegraphics{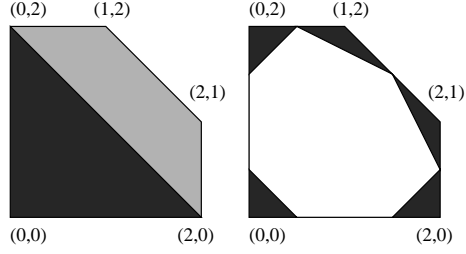}
\caption{Polytope corresponding to the Delzant manifold $(M, \, \sigma)$
obtained by blowing up $\mathbb{S}^2_{r_0} \times \mathbb{S}^2_{r_0}$ with $r_0=1/\sqrt{2}$. 
Observe that $\textup{Emb}_{(M,\,\sigma)}(\sqrt{2})=0$ (proof in left figure) 
and $\textup{Emb}_{(M,\,\sigma)}(1/\sqrt{2})=10$ (proof in right figure). See Lemma \ref{t}.}
\end{center}
\label{AF6}
\end{figure}

As a final remark we observe that 
the invariant function $\textup{Emb}_{(M,\,\sigma)}$ associated to the Delzant manifold $(M, \, \sigma)$
always reaches its minimum and maximum values
on an interval of strictly positive length. 

\begin{lemma} \label{t}
There exist numbers $r_0,\, s_0>0$ such that if $r \le r_0$, then the space of equivariant symplectic embeddings from 
$\mathbb{B}_r$ into $M$ is homotopically
equivalent to a disjoint union of $n! \, \chi(M)$ copies of $\mathbb{T}^n$,
and if $s \ge s_0$, then it is empty.
\end{lemma}

\begin{proof}
It follows easily from Lemma \ref{ss}, Corollary \ref{lastc} and the previous observations.
\end{proof}

This concludes the proof of Theorem \ref{maintheorem} (and hence by construction the proof of Proposition \ref{maintheorem2}).

\section{Remarks on the partially equivariant case of Theorem \ref{maintheorem}}

In this section we initiate a discussion on the topology of the space of \emph{partially
equivariant} symplectic embeddings and sketch some suggestions to answer a question in this 
direction.

First the notion of $\Lambda$\--equivariance ($\Lambda \in \mbox{Aut}(\mathbb{T}^n)$)
in Section 1 has a natural extension: we say that an embedding
from the $2n$\--ball $\mathbb{B}_r$ into the $2n$\--dimensional Delzant manifold $M$ is \emph{$\gamma$\--equivariant with respect to
a monomorphism $\gamma \colon \mathbb{T}^{n-k} \to \mathbb{T}^n$,
$1 \le k \le n-1$}, if the following diagram commutes:
$$
\xymatrix{ \ar @{} [dr] |{\circlearrowleft}
\mathbb{T}^{n-k} \times \mathbb{B}_r  \ar[r]^{\gamma \times f}      \ar[d]^{  {\cdot} }  &  \mathbb{T}^{n} \times M 
                  \ar[d]^{\psi}   \\
                   \mathbb{B}_r  \ar[r]^f   &       M}.
$$
For example, $M^{\gamma}$ is the set of $ {p} \in M$ such that $\psi(\gamma(t), \, p)= {p}$
for all $t \in \mathbb{T}^{n-k}$, and the rest of terminology
is also analogous. This definition extends naturally to the case when $k=n$, in which
the embeddings considered are \emph{purely symplectic}, 
as well as to the case when $k=0$, in which the embeddings are
\emph{fully equivariant}, case which we treated previously in the paper. Unless otherwise
specified we do not consider these two cases in the discussion that follows.
The question we would like to address is the following:

\begin{question} \label{t2}
Let $r$ be such that any connected component $C$ of $M^{\gamma}$ admits a Darboux--Weinstein
neighborhood of radius $r$, and by this we mean a neighborhood that is
equivariantly symplectomorphic to a bundle over $C$ with fiber the standard
ball of radius $r$. 
Is the space of $\gamma$\--equivariant symplectic embeddings from $\mathbb{B}_r$
into $M$ homotopically equivalent to the space of purely symplectic embeddings
from $\mathbb{B}_r^{2k}$ into $M^{\gamma}$ up to reparametrization groups
\textup{(}as explained below\textup{)}?
\end{question}

To analyze Question \ref{t2} first define $\widehat {\mathbb{B}}^{2k}_r$ to be the embedded 
$2k$\--ball in $\mathbb{B}_r$, i.e. the set of points $(z_1,\ldots,z_{k}, \, {0})$ in $\mathbb{B}_r$ so that 
$\sum_{i=1}^{k} |z_i|^2 \le r^2$. The preimage under the momentum map of the $k$\--face 
corresponding to $\gamma$ is the fixed point locus $M^{\gamma}$. 
Now consider any symplectic embedding $f \colon \widehat {\mathbb{B}}^{2k}_r \to M^{\gamma}$. 
We want to find a canonical way to extend $f$ to an equivariant symplectic 
embedding $\mbox{can}(f) \colon \mathbb{B}_r \to M$
up to homotopy. 

Here is an attempt to construct $\mbox{can}(f)$: near the image of $f$, we can apply the equivariant version of the Darboux--Weinstein's
theorem in order to find a neighborhood of $\mbox{Im}(f)$ in $M$ which is symplectomorphic to 
$\mbox{Im}(f) \times \mathbb{B}^{2(n-k)}$, with 
the action of $\mathbb{T}^{n-k}$ given by the standard action 
on $\mathbb{B}^{2(n-k)}$, and the symplectic form coinciding
with the product symplectic form. Note that the symplectic normal bundle to 
$M^\gamma$ is trivial over $\mbox{Im}(f)$ because $\mbox{Im}(f)$ is contractible, 
so a neighborhood of $\mbox{Im}(f)$ looks like $\mbox{Im}(f) \times \mathbb{B}^{2(n-k)}$ 
with a product symplectic form, and the action of $\mathbb{T}^{n-k}$ on 
it is conjugate to the standard one. Using this identification $M$ is described as
a product, and we can define
$\mbox{can}(f)( {z}):=(f(z_1,\ldots,z_{k}),\, z_{k+1}, \ldots , z_n)$. This 
expression for $\mbox{can}(f)$ is clearly symplectic and equivariant with respect to 
$\mathbb{T}^{n-k}$\--actions on the last $n-k$ coordinates but is not canonical because the local
symplectomorphism given by Darboux--Weinstein's theorem is not unique. 
We cannot expect it to always be the same 
independently of $f$, because it is not true 
that globally the normal bundle to $M^\gamma$ is symplectically trivial, it
only becomes true over a neighborhood of $\mbox{Im}(f)$. So this construction
depends on choices of parameters. 

Calling  $\textup{CAN}$ the space of canonical embeddings $\mbox{can}(f) \colon \mathbb{B}_r \to
 \Delta^M$,
where $f \colon \widehat {\mathbb{B}}^{2k}_r  \to M^{\gamma}$ is a symplectic embedding,
observe that $\textup{CAN}$ is naturally identified with the space of purely symplectic 
embeddings from the standard $\mathbb{B}^{2k}_r$ into $M^{\gamma}$, up to homotopy.
The question then becomes whether any $\gamma$\--equivariant symplectic embedding
$f \colon \mathbb{B}^{2k}_r \to M$ may be deformed through a continuous family
of equivariant symplectic embeddings to an embedding in $\textup{CAN}$.

Equivalently, we ask the question: is the natural map between the space of partially equivariant embeddings
from $\mathbb{B}_{r}$ into $M$ and the space of symplectic embeddings from $\mathbb{B}^{2k}_{r}$
into the fixed point set $M^{\gamma}$ (given by the restriction to the fixed ball
$\mathbb{B}_{r}^{2k}$) a fibration? Note that the construction of $\textup{can}(f)$
would give a section of this fibration.

\begin{conjecture}
Question \textup{3.1} has an affirmative answer.
\end{conjecture}

\begin{example}
\normalfont
If $M=\mathbb{S}^2 \times \mathbb{S}^2$ 
with a product symplectic form and product $\mathbb{T}^2$\--action 
(this space has been carefully studied by Lalonde--Pinsonnault \cite{LP} and Anjos \cite{An} among
other authors), the
fixed point locus of the second $\mathbb{S}^1$ factor is
$\mathbb{S}^2 \times \{a,\, b\}$, 
where $a,b$ are
the fixed points of the action of $\mathbb{S}^1$ on $\mathbb{S}^2$. Now, given a symplectic
embedding $f$ of the ball $\mathbb{B}^2$ into $\mathbb{S}^2$, it is easy to build an $\mathbb{S}^1$\--equivariant
embedding of $\mathbb{B}^4$ into $\mathbb{S}^2 \times \mathbb{S}^2$ 
canonically by $(z_1,\,z_2) \mapsto (f(z_1),\,z_2)$,
where $z_2$ is taken to be a coordinate centered at the fixed point $a$. 
In this case the normal bundle to the fixed point component $\mathbb{S}^2 \times \{a\}$ is globally trivial.
\end{example}

\begin{rm}
\normalfont
The combination of purely symplectic results of Biran, Lalonde--Pinsonnault and others and an affirmative answer to
Question 3.1 would give insight
into the partially equivariant case in higher dimensions; for example
McDuff showed that if $M$ is a symplectic $4$-manifold with non-simple
Seiberg--Witten type, then the space of symplectic embeddings
from $\mathbb{B}_r$ into $M$ is path connected (which extends results of Biran). This is
a consequence of the non--trivial result: \emph{any two cohomologous and deformation equivalent symplectic forms 
on $M$ are isotopic} (proved in \cite{M1}). Examples are known in dimensions
6 and above of cohomologous symplectic forms that are deformation equivalent but not isotopic,
so these techniques do not help to understand the topology
of the space of symplectic embeddings from $\mathbb{B}_r$ into $M$. A positive answer to Question \ref{t2} would give
the first non--trivial result in dimension $6$. 

Another way of trying to
generalize
Theorem \ref{maintheorem} is to consider 
embeddings equivariant with respect to a complexity 
one action, that is, an action of $\mathbb{T}^{n-1}$ 
on $M^{2n}$.  This is a hopeful approach 
since a complete classification of complexity one actions has 
been recently achieved by Karshon and Tolman \cite{KT}.
\end{rm}

\section*{Acknowledgments} The author is grateful to D. Auroux and 
V. Guillemin for discussions, and for hosting  him 
at the M.I.T. regularly during the Fall and Spring semesters of 2003 and 2004.   
He thanks D. Auroux, J.J. Duistermaat, Y. Karshon and M. Pinsonnault for making comments 
on a preliminary version of this paper. Finally, the author is grateful to an anonymous  
referee for helpful suggestions that have shortened the proof in Section 2, as well as for his/her 
interesting comments on Section 3.

\noindent
A. Pelayo\\
Department of Mathematics, University of Michigan\\
2074 East Hall, 530 Church Street, Ann Arbor, MI 48109--1043, USA\\
e\--mail: apelayo@umich.edu

\end{document}